\documentclass[10pt]{amsart}
\usepackage{amssymb,amsmath,amsthm}
\usepackage[colorlinks=true]{hyperref}
\usepackage{mathtools}
\usepackage{enumitem}

\usepackage{tikz}
\usetikzlibrary{arrows,decorations.pathmorphing,backgrounds,positioning,fit,petri}
\usetikzlibrary{decorations.markings}
\tikzset{->-/.style={decoration={
  markings,
  mark=at position .5 with {\arrow{>}}},postaction={decorate}}}

\DeclarePairedDelimiterX{\norm}[1]{\lVert}{\rVert}{#1}

\usepackage{cleveref}
\crefformat{section}{\S#2#1#3}

\theoremstyle{plain}
\newtheorem{theorem}{Theorem}[section]

\newtheorem{lemma}[theorem]{Lemma}
\newtheorem{cor}[theorem]{Corollary}
\newtheorem{prop}[theorem]{Proposition}
\theoremstyle{definition}

\begin{document}
\title{Isoperimetric inequality for Disconnected Regions}

\author{Bidyut Sanki}
\address{Department of Mathematics and Statistics, Indian Institute of Technology\\ Kanpur\\ Uttar Pradesh - 208016\\ India}
\email{bidyut@iitk.ac.in}
\author{Arya Vadnere}
\address{Department of Mathematics, University of Buffalo \\ New York - 14260-2900 \\ United States of America}
\email{aryaabhi@buffalo.edu}

\maketitle

\begin{abstract}
The discrete isoperimetric inequality in Euclidean geometry states that among all $n$-gons having a fixed perimeter $p$, the one with the largest area is the regular $n$-gon. The statement is true in spherical geometry (see~\cite{Las}) and hyperbolic geometry (see~\cite{Bez}) as well. 

In this paper, we generalize the discrete isoperimetric inequality to disconnected regions, i.e. we allow the area to be split between regions. We give necessary and sufficient conditions for the result (in Euclidean, spherical and hyperbolic geometry) to hold for multiple $n$-gons whose areas add up. 
\end{abstract}


\section{Introduction} \label{sec:intro}

The discrete isoperimetric inequality is a profound result in Euclidean geometry, which states that among all $n-$gons with fixed area $a>0$, the one with the least perimeter is the regular $n-$gon. The proof is a standard exercise in Euclidean geometry. The similar statement also hold in spherical geometry (as shown in \cite{Las} by L\'{a}szl\'{o} Fejes T\'{o}th) and in hyperbolic geometry (as shown in \cite{Bez} by K\'{a}roly Bezdek). 

In~\cite{Csi}, Csik\'{o}s, L\'{a}ngi and Nasz\'{o}di have extended this result from geodesic polygons to polygons bounded by arcs of constant geodesic curvature $k_{g}$, in Euclidean, spherical and hyperbolic geometry. A similar optimization problem, regarding maximizing the area of polygons with bounded diameter, has been dealt with by Bieberbach (see \cite{Bie}) in the Euclidean geometry, and by B\"{o}r\"{o}czky and Sagmeister (see \cite{Bor}) in the spherical and hyperbolic geometry.  

In this paper, we consider the generalization of the discrete isoperimetric inequality to multiple polygons with fixed total area. More precisely, let $\mathbb{M}$ denote either the Euclidean plane $\mathbb{R}^2$, the Riemann sphere $\mathbb{S}^2$, or the hyperbolic plane $\mathbb{H}^2$ with the respective geometry of constant sectional curvature $K=0,1$ or $-1$. Define a \textit{configuration} of polygons in $\mathbb{M}$ as a finite set of disjoint polygons, $P_1, \dots, P_k,$ in $\mathbb{M}$. We define the \textit{total area} (similarly, the \textit{total perimeter}) of a configuration of polygons as the sum of their areas, i.e., $\sum\limits_{i=1}^k \mathrm{area}(P_i)$ (similarly, the perimeters, i.e., $\sum\limits_{i=1}^k \mathrm{perim}(P_i)$). For simplicity, we only work with non-degenerate polygons (i.e. polygons with non-zero area). 

Suppose $P$ is a regular $n$-sided polygon and $\{P_1, \dots, P_k\}$ is an arbitrary configuration of $n$-sided polygons in $\mathbb{M}$ with the same total area as $P$. The central question we study here is finding under what circumstances the inequality below $$\mathrm{perim}(P)\leq \sum\limits_{i=1}^k \mathrm{perim}(P_i)$$ holds true in $\mathbb{M}.$ 

Note that, by the isoperimetric inequality, any configuration achieving minimal total perimeter would only have regular $n$-gons. Thus, we shall restrict our attention to configurations with only regular $n$-gons throughout this paper, and note that the results would carry over for general configurations.

We prove the following result to answer the question above, in the cases when $\mathbb{M}=\mathbb{R}^2$ and $\mathbb{S}^2$.

\begin{prop}\label{prop:weak}
Let $P_1, \dots, P_k$ and $P$ be regular $n$-gons (for $n\geq3$) in $\mathbb{M}$, where $\mathbb{M}$ is either $\mathbb{R}^2$ or $\mathbb{S}^2$, with areas $a_1, \dots, a_k$ and $a$ respectively, satisfying $\sum\limits_{i=1}^k a_i = a$. Then 
\begin{equation}\label{ineq:main}
\mathrm{perim}(P) \leq \sum\limits_{i=1}^k \mathrm{perim}\left(P_i\right). 
\end{equation}
\end{prop}
Thus, among all configurations with fixed total area, the configuration with a single regular $n$-gon is the unique configuration achieving minimal total perimeter, when $\mathbb{M}$ is either $\mathbb{R}^2$ or $\mathbb{S}^2$.

Proposition~\ref{prop:weak} is not true in general when $\mathbb{M}$ is the hyperbolic plane. We can have hyperbolic polygons with bounded area but arbitrarily large perimeter. For a counter-example, let $T_{\epsilon}$ denote the regular hyperbolic triangle with interior angle $\epsilon$, where $0< \epsilon< \frac{\pi}{3}$. By Gauss-Bonnet theorem (see Chapter $7$, Theorem $6.4$ in ~\cite{ONeill}), $$\mathrm{area}\left(T_{\epsilon}\right)=\pi-3\epsilon.$$ Now, consider the configuration $\left\{ T_{1},T_{2}\right\} $ consisting of two regular hyperbolic triangles with areas $\frac{\pi}{2},\frac{\pi}{2}-3\epsilon$ respectively, so that $$\mathrm{area}(T_\epsilon) = \mathrm{area}(T_1)+\mathrm{area}(T_2).$$ Now, using hyperbolic trigonometric formula (see Theorem 2.2.1, (ii) in~\cite{Bus}), we have 
\begin{align*}
\mathrm{perim(T_1)} &= 3 \cosh^{-1}(3+2\sqrt{3}) \text{ and }\\ \mathrm{perim(T_2)} &=3 \cosh^{-1}\left( \frac{\cos^2\left(\frac{\pi}{6}+\epsilon \right) + \cos\left(\frac{\pi}{6}+\epsilon \right) }{\sin^2\left(\frac{\pi}{6}+\epsilon \right)}\right).
\end{align*}
Note that $\mathrm{perim(T_2)}$ is bounded above by $3 \cosh^{-1}(3+2\sqrt{3})$. Therefore, we have $$\mathrm{perim}(T_1)+ \mathrm{perim(T_2)}\leq 6 \cosh^{-1}(3+2\sqrt{3}).$$ On the other hand $\mathrm{perim}(T_{\epsilon})$ is a continuous function in $\epsilon$ and  
\[
\lim_{\epsilon\to0}\mathrm{perim}\left(T_{\epsilon}\right)=+\infty.
\]
Thus, there exists a value $\epsilon$ small enough, so that 
\[
\mathrm{perim}\left(T_{1}\right)+\mathrm{perim}\left(T_{2}\right)<\mathrm{perim}\left(T_\epsilon\right).
\]

We prove the theorem below:

\begin{theorem} \label{thm:Hyp}
Let $P$ be a regular hyperbolic $n$-gon (for $n\geq3$) with interior angle $\theta$. There exists a constant $\Theta=\Theta(n)$, depending only on $n$, such that 
\begin{enumerate}
\item If $\theta\geq\Theta$, then for any configuration of $n$-gons $P_{1},\dots,P_{k},$ with total area equal to $\mathrm{area}(P)$, we have $$ \mathrm{perim}\left(P\right)\leq \sum_{i=1}^{k}\mathrm{perim}\left(P_{i}\right).$$
Furthermore, if $\theta>\Theta$, then equality occurs only if $k=1$.
\item If $\theta<\Theta$, then there exists a configuration $Q_{1},\dots,Q_{l},$ (for some $l\geq1$) of $n$-gons, with total area equal to $\mathrm{area}(P)$, satisfying $$ \sum_{i=1}^{l}\mathrm{perim}\left(Q_{i}\right)<\mathrm{perim}\left(P\right).$$
\end{enumerate}
\end{theorem}

In short, the isoperimetric inequality for multiple disconnected hyperbolic polygons would hold only for sets of configurations whose total area is bounded above by $(n-2)\pi-n\Theta$, where $\Theta$ is the fixed real number depending on $n$, as given in Theorem~\ref{thm:Hyp}.

We shall prove Proposition \ref{prop:weak} in section \ref{sec:Euc_Sph}, and deal with the case when $\mathbb{M}=\mathbb{H}^2$ in section \ref{sec:Hyp_Geo}. Finally, we propose further directions of enquiry in section \ref{sec:Further}.


\section{Euclidean and Spherical Geometry}\label{sec:Euc_Sph}

\subsection{Euclidean Geometry}
We shall first prove Proposition \ref{prop:weak} for $\mathbb{M}=\mathbb{R}^2$ and configurations with at most two polygons. Then we deduce the general result as a corollary.

\begin{theorem} \label{thm:Euc}
Suppose $P_1,P_{2}$ and $P$ are regular Euclidean $n$-gons (for $n\geq3$), with areas $a_1,a_{2}$ and $a$ respectively, satisfying $a_{1}+a_{2}=a$. Then $$ \mathrm{perim}\left(P\right) \leq
\mathrm{perim}\left(P_{1}\right)+\mathrm{perim}\left(P_{2}\right),$$
with equality if and only if at least one of $P_{1}$ and $P_{2}$ is degenerate $($i.e. one of $a_{1}$ or $a_{2}$ is $0)$. 
\end{theorem}

\begin{proof}
For a regular $n$-gon with perimeter $p$ and area $a$, we have the relation $$a=\frac{p^{2}}{4n\tan\left(\pi/n\right)}=c p^2,$$ where $c=\frac{1}{4n\tan(\pi/n)}$ is a constant depending only on $n$. Thus, using the equality $a_1+a_2=a$, we have
\begin{align*}
\left(\mathrm{perim}\left(P_{1}\right)\right)^{2}+\left(\mathrm{perim}\left(P_{2}\right)\right)^{2} & =\left(\mathrm{perim}\left(P\right)\right)^{2}.
\end{align*}
By Pythagorus' theorem, we get that $\mathrm{perim}\left(P_{1}\right)$, $\mathrm{perim}\left(P_{2}\right)$ and $\mathrm{perim}\left(P\right)$ form the sides of a right angled triangle with $\mathrm{perim}(P)$ being the hypotenuse. The triangle inequality then implies that
\[
\mathrm{perim}\left(P\right)\leq \mathrm{perim}\left(P_{1}\right)+\mathrm{perim}\left(P_{2}\right),
\]
and the inequality is strict unless the right angled triangle is degenerate. This means that either $\mathrm{perim}\left(P_{1}\right)$ or $\mathrm{perim}\left(P_{2}\right)$ is $0$.  
\end{proof}

Now, in the Euclidean case, Proposition \ref{prop:weak} follows as an immediate corollary. 
\begin{cor}
Proposition \ref{prop:weak} is true when $\mathbb{M}=\mathbb{R}^{2}$. In particular, equality occurs in the inequality~\eqref{ineq:main} only when $k=1$. 
\end{cor}
\begin{proof}
By applying Theorem~\ref{thm:Euc} to pairs of polygons and proceeding recursively, we have the corollary.
\end{proof}

\subsection{Spherical Geometry}

We shall now proceed similarly to prove Proposition \ref{prop:weak} for $\mathbb{M}=\mathbb{S}^2$ and configurations with at most two polygons. Then we deduce the general result as a corollary. Before proceed, we recall the following lemma from analysis which is used in the proof. 

\begin{lemma}\label{lem:ineq}
Let $f:[a,b]\to\mathbb{R}$ be a strictly concave function. Suppose $c,d\in\left(a,b\right)$
such that $a+b=c+d$. Then $$f(c)+f(d)> f(a)+f(b).$$
\end{lemma}
\begin{proof}
By concavity of the function $f$, the points $(c,f(c))$ and $(d,f(d))$ on the graph of $f$ lie above the chord joining points $(a,f(a))$ and $(b,f(b))$. Therefore, we have
\begin{align*} 
f(c) & >f(a)+\frac{f(b)-f(a)}{b-a}(c-a)\text{ and }\\ f(d) & > f(b)+\frac{f(a)-f(b)}{a-b}(d-b).
\end{align*}
Adding these two inequalities, we have the lemma. 
\end{proof}
Now, note that by Gauss-Bonnet theorem, the area of a regular spherical $n$-gon, with interior angle $\theta$, is $n\theta-\left(n-2\right)\pi$. Thus, the interior angle $\theta$ of a regular spherical $n$-gon must satisfy $\frac{(n-2)\pi}{n}< \theta< \pi$. With that in mind, we prove the following theorem.
\begin{theorem} \label{thm:Sph}
Let $P_{1},P_{2}$ and $P$ be regular spherical $n$-gons (for $n\geq3$) with areas $a_{1},a_{2}$ and $a$ respectively. Suppose that $a_{1}+a_{2}=a$. Then 
$$\mathrm{perim}\left(P\right)< \mathrm{perim}\left(P_{1}\right)+\mathrm{perim}\left(P_{2}\right).$$

\end{theorem}

\begin{proof}
Consider a triangulated section $\triangle OAM$ of the polygon, as shown in Figure \ref{fig:sph_triangle}. 
\begin{figure}[htbp] 
\centering
\includegraphics[scale=0.5]{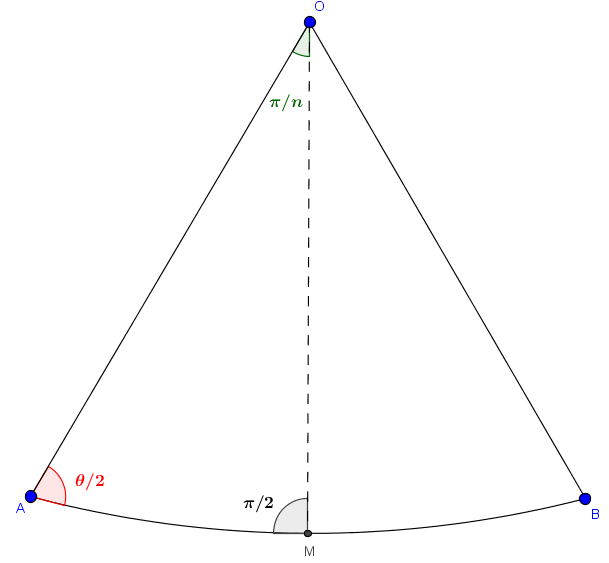}

\caption{A triangular section of a regular spherical $n$-gon. Here, $A,B$ are two consecutive vertices of the polygon, $O$ is the circumcenter and $M$ is the midpoint of $\overline{AB}$.}
\label{fig:sph_triangle}
\end{figure}
Using spherical trigonometric identities (see point $(62)$ in \cite{Tod}) to $\triangle OAM$, we get that 
\[
\cos\left(\frac{s}{2}\right)=\frac{\cos\left(\frac{\pi}{n}\right)}{\sin\left(\frac{\theta}{2}\right)},
\]
where $s$ equals the length of a side of the polygon. Thus, the perimeter of a regular $n$-gon with interior angle $\theta$ equals $$\mathrm{perim}(P) = 2n\arccos\left(\frac{\cos\left(\frac{\pi}{n}\right)}{\sin\left(\frac{\theta}{2}\right)}\right).$$

Now, let $\theta_{1},\theta_{2}$ and $\theta$ denote the interior angles of $P_{1},P_{2}$ and $P$ respectively. Then the condition $a_{1}+a_{2}=a$ translates to $$
\theta_{1}+\theta_{2}=\theta+\frac{(n-2)\pi}{n}.
$$
Therefore, to prove the inequality $
\mathrm{perim} \left( P \right) < \mathrm{perim}\left(P_{1}\right)+\mathrm{perim}\left(P_{2}\right)$, it suffices to prove
$$f\left(\theta\right)< f\left(\theta_{1}\right)+f\left(\theta_{2}\right) ,
$$
where $f:\left[\frac{(n-2)\pi}{n},\pi\right]\to\mathbb{R}$ is the function, defined by $$
f(x)=\arccos\left(\frac{\cos\left(\frac{\pi}{n}\right)}{\sin\left(\frac{x}{2}\right)}\right).$$
To prove this, we first show that $f$ is strictly concave in the domain $\left[\frac{(n-2)\pi}{n},\pi\right]$ which is the domain for an interior angle of a regular spherical $n$-gon. We include the degenerate case to simplify the analysis. To this extent, we compute that 
\begin{align*}
f''(x)= & -\frac{\cos\left(\frac{\pi}{n}\right)\csc^{3}\left(\frac{x}{2}\right)}{4\sqrt{1-\cos^{2}\left(\frac{\pi}{n}\right)\csc^{2}\left(\frac{x}{2}\right)}}-\frac{\cos\left(\frac{\pi}{n}\right)\cot^{2}\left(\frac{x}{2}\right)\csc\left(\frac{x}{2}\right)}{4\sqrt{1-\cos^{2}\left(\frac{\pi}{n}\right)\csc^{2}\left(\frac{x}{2}\right)}}\\
 & -\frac{\cos^{3}\left(\frac{\pi}{n}\right)\cot^{2}\left(\frac{x}{2}\right)\csc^{3}\left(\frac{x}{2}\right)}{4\left(1-\cos^{2}\left(\frac{\pi}{n}\right)\csc^{2}\left(\frac{x}{2}\right)\right)^{3/2}}.
\end{align*}
Since each term in this sum is negative when $x\in\left(\frac{(n-2)\pi}{n},\pi\right)$, we see that $f''(x)<0$ for all $x\in\left(\frac{(n-2)\pi}{n},\pi\right)$. Thus, $f$ is strictly concave in $\left[\frac{(n-2)\pi}{n},\pi\right]$. Now, the conclusion follows from Lemma~\ref{lem:ineq}. 
\end{proof}

We conclude this section, by noting that Proposition \ref{prop:weak} follows as a direct corollary.

\begin{cor}
Proposition \ref{prop:weak} holds when $\mathbb{M}=\mathbb{S}^{2}$. In particular, inequality \eqref{ineq:main} is an equality only in the case when $k=1$. 
\end{cor}
\begin{proof}
By applying Theorem \ref{thm:Sph} recursively to pairs of polygons, we have the corollary. 
\end{proof}

\section{Hyperbolic Geometry}\label{sec:Hyp_Geo}


The aim of this section is to prove Theorem \ref{thm:Hyp}. The strategy we use here is similar to the proof of Theorem \ref{thm:Sph}. 

Let $P$ be a regular hyperbolic $n$-gon with interior angle $\theta>0$. By the Gauss-Bonnet theorem, we have that 
\[
\mathrm{area}\left(P\right)=\left(n-2\right)\pi-n\theta,
\]
so that $\theta\in\left(0,\frac{(n-2)\pi}{n}\right)$. To compute the perimeter of $P$, consider a triangular section of $P$ as shown in Figure \ref{fig:hyptriangle}.

\begin{figure}[htbp] 
\centering
\includegraphics[scale=1]{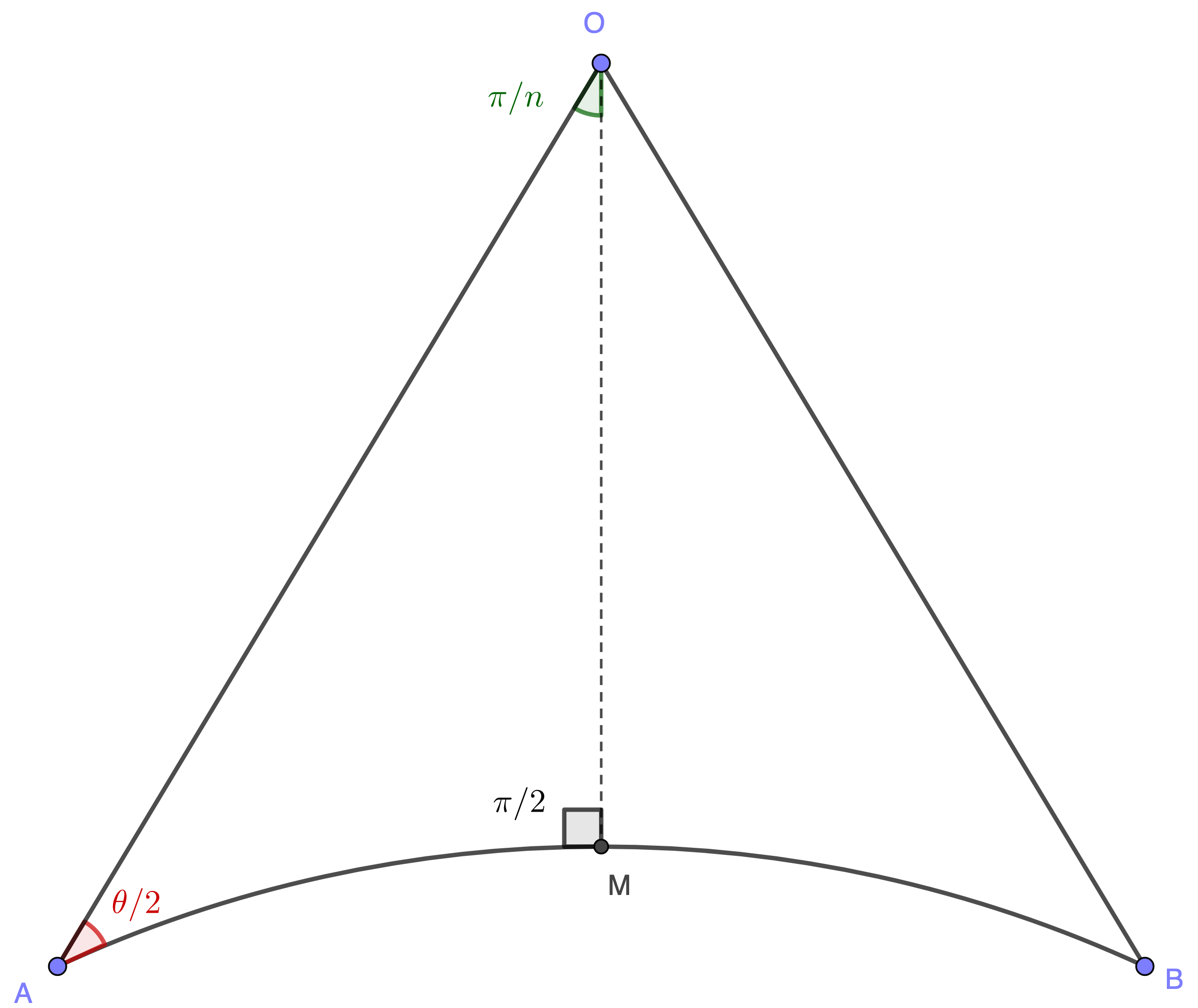}

\caption{A triangular section of a regular hyperbolic $n$-gon. Here, $A,B$ are two consecutive vertices of the polygon, $O$ is the circumcenter and $M$ is the midpoint of $\overline{AB}$.}
\label{fig:hyptriangle}
\end{figure}

Using hyperbolic trigonometric identities (see formula (v) in Theorem 2.2.2~\cite{Bus}) to $\triangle OAM$, we get that $$\cosh \left( \frac{s}{2} \right) = \frac{\cos\left(\frac{\pi}{n}\right)}{\sin\left(\frac{\theta}{2}\right)},$$ where $s$ is the length of a side of the regular $n$-gon $P$. As a result, $$ \mathrm{perim}\left(P\right)=2n\cosh^{-1}\left(\frac{\cos\left(\frac{\pi}{n}\right)}{\sin\left(\frac{\theta}{2}\right)}\right).$$

Similar to the Euclidean and spherical cases, we would first prove the result for configurations of at most two polygons.
\begin{theorem} \label{thm:Hyp_two}
Let $P$ be a regular hyperbolic $n$-gon (for $n\geq3$) with interior angle $\theta$. Then there exists a constant $\Theta=\Theta(n)$, depending only on $n$, such that 
\begin{enumerate}
\item If $\theta\geq\Theta$, then for any configuration of polygons $P_{1},P_2$ with total area equal to $\mathrm{area}(P)$, we have
$$ \mathrm{perim}\left(P\right)\leq \mathrm{perim}\left(P_1\right)+\mathrm{perim}\left(P_2\right).$$
If $\theta>\Theta$, then the inequality is strict. 
\item If $\theta<\Theta$, then there exists a configuration of polygons $Q_{1},Q_{2}$ with total area equal to $\mathrm{area}(P)$, satisfying 
\[
\mathrm{perim}\left(Q_1\right)+\mathrm{perim}\left(Q_2\right)<\mathrm{perim}\left(P\right).
\]
\end{enumerate}
\end{theorem}

Before proceed to the proof of Theorem~\ref{thm:Hyp_two}, we develop some technical lemmas.

\subsection{Technical Lemmas}
 For the remainder of this section, we define the function $g_{n}:\left(0,\frac{(n-2)\pi}{n}\right)\to\mathbb{R}$ for each $n\geq3$, given by \begin{equation}\label{eq:g}
 g_{n}(x)=\cosh^{-1}\left(\frac{\cos\left(\frac{\pi}{n}\right)}{\sin\left(\frac{x}{2}\right)}\right), 
\end{equation} so that $\mathrm{perim}\left(P\right)=2n\cdot g_{n}\left(\theta\right)$. The proof of Theorem \ref{thm:Hyp_two} involves analysing the function $g_n$, and now we set up the technical lemmas in this regard. 
\begin{lemma} \label{lem:der_convex}
For any $n\geq3$, the derivative $g'_{n}:\left(0,\frac{(n-2)\pi}{n}\right)\to\mathbb{R}$ of $g_{n}$ is strictly concave over its domain. 
\end{lemma}

\begin{proof}
Fixing $n\geq3$, we first we compute that 
\begin{align*}
g'_{n}(x) & =-\frac{\cos\left(\frac{\pi}{n}\right)\cot\left(\frac{x}{2}\right)}{2\sqrt{\cos^{2}\left(\frac{\pi}{n}\right)-\sin^{2}\left(\frac{x}{2}\right)}}\\
 & =-\left(\frac{\cos\left(\frac{\pi}{n}\right)}{\sqrt{2}}\right)\frac{\cot\left(\frac{x}{2}\right)}{\sqrt{\cos\left(\frac{2\pi}{n}\right)+\cos(x)}}.
\end{align*}
Then 
\[
g''_{n}(x)\cdot\left(\frac{2\sqrt{2}}{\cos\left(\frac{\pi}{n}\right)}\right)=\frac{\csc^{2}\left(\frac{x}{2}\right)}{\sqrt{\cos\left(\frac{2\pi}{n}\right)+\cos(x)}}-\frac{\sin(x)\cot\left(\frac{x}{2}\right)}{\left(\cos\left(\frac{2\pi}{n}\right)+\cos(x)\right)^{3/2}},
\]
and 
\begin{align*}
g'''_{n}(x)\cdot\left(\frac{4\sqrt{2}}{\cos\left(\frac{\pi}{n}\right)}\right)= & -\frac{\cot\left(\frac{x}{2}\right)\csc^{2}\left(\frac{x}{2}\right)\left[2\cos\left(\frac{2\pi}{n}\right)+3\cos(x)-1\right]}{\left(\cos\left(\frac{2\pi}{n}\right)+\cos(x)\right)^{3/2}}\\
 & -\frac{\sin(x)\left[\cos(x)+3-2\cos\left(\frac{2\pi}{n}\right)\right]}{\left(\cos\left(\frac{2\pi}{n}\right)+\cos(x)\right)^{5/2}}.
\end{align*}

We need to show that $g'''_{n}(x)<0$ for each $x\in\left(0,\frac{(n-2)\pi}{n}\right)$. Since $\cos\left(\frac{2\pi}{n}\right)+\cos(x)>0$ in the domain, this is equivalent to show that
\begin{align*}
\cot\left(\frac{x}{2}\right)\csc^{2}\left(\frac{x}{2}\right)\left[2\cos\left(\frac{2\pi}{n}\right)+3\cos(x)-1\right]\left[\cos\left(\frac{2\pi}{n}\right)+\cos(x)\right]\\
+\sin(x)\left[\cos(x)+3-2\cos\left(\frac{2\pi}{n}\right)\right] & >0.
\end{align*}

For $n\geq3$, we have $$ \cos(x)+3-2\cos\left(\frac{2\pi}{n}\right)>1+\cos(x).$$
Since $\cot\left(\frac{x}{2}\right)\csc^{2}\left(\frac{x}{2}\right)>0$ in the domain, it suffices to show that 
\begin{align}
\left[1-2\cos\left(\frac{2\pi}{n}\right)-3\cos(x)\right] \left[\cos\left(\frac{2\pi}{n}\right)+\cos(x)\right] & <\sin^{2}(x)\sin^{2}\left(\frac{x}{2}\right).\label{ineq:trip_der}
\end{align}

Now, the right hand side of inequality~\eqref{ineq:trip_der} is always positive in the domain. The left hand side of inequality~\eqref{ineq:trip_der} is positive if and only if 
\begin{equation}\label{ineq:cos(x)}
\cos(x)<\frac{1-2\cos\left(\frac{2\pi}{n}\right)}{3}.
\end{equation}
We note that, when $\cos(x)$ is greater than or equal to $\frac{1-2\cos\left(\frac{2\pi}{n}\right)}{3}$, inequality~\eqref{ineq:trip_der} would hold directly. Now, when inequality \eqref{ineq:cos(x)} is indeed satisfied, from inequality~\eqref{ineq:trip_der}, we get that
\begin{align*}
\sin^{2}(x)\sin^{2}\left(\frac{x}{2}\right) & >\frac{1}{2}\left[\frac{4-2\cos\left(\frac{2\pi}{n}\right)}{3}\right]\left[\frac{2+2\cos\left(\frac{2\pi}{n}\right)}{3}\right]^{2}\\
 & =\frac{4}{9}\left[2-\cos\left(\frac{2\pi}{n}\right)\right]\left[1+\cos\left(\frac{2\pi}{n}\right)\right]^{2}.
\end{align*}
We also have 
\[
\cos\left(\frac{2\pi}{n}\right)+\cos(x)<\frac{1+\cos\left(\frac{2\pi}{n}\right)}{3},
\]
and for $x\in\left(0,\frac{(n-2)\pi}{n}\right)$,
\[
-3\cos(x)<-3\cos\left(\frac{(n-2)\pi}{n}\right)=3\cos\left(\frac{2\pi}{n}\right).
\]

Thus, it suffices to show that 
\[
1+\cos\left(\frac{2\pi}{n}\right)<\frac{4}{3}\left[1+\cos\left(\frac{2\pi}{n}\right)\right]\left[2-\cos\left(\frac{2\pi}{n}\right)\right],
\]
which is clearly true for any $n\geq3$. 
\end{proof}

\begin{lemma} \label{lem:local_min}
Let $c\in\left( \frac{(n-2)\pi}{n},\frac{(2n-4)\pi}{n}\right)$ be a fixed constant. Consider the function $h_{n}:\left(c-\frac{(n-2)\pi}{n},\frac{(n-2)\pi}{n}\right),$ defined by $$
h_{n}(x)=g_{n}(x)+g_{n}(c-x).$$
Then the only possible local minimum for $h_{n}$ in its domain is at $x=\frac{c}{2}$. 
\end{lemma}

\begin{proof}
By differentiating the function $h_n$, we have $$ h_n'(x) = g'_{n} (x) - g'_{n} (c-x).$$
Therefore, at a point $x$ of local extremum of $h_{n}$, we have 
\begin{equation}\label{eq:der_roots}
g_{n}'(x)=g_{n}'(c-x).
\end{equation}
Since $x=\frac{c}{2}$ is an obvious solution to equation~\eqref{eq:der_roots}, this is one possible candidate for a local minimum. Similarly, if $x=x_{0}$ is a solution to equation~\eqref{eq:der_roots}, then so is $x=c-x_{0}$. Therefore, it suffices to check for possible solutions to equation~\eqref{eq:der_roots} in the interval $\left(c-\frac{(n-2)\pi}{n},\frac{c}{2}\right)$. 

We claim that there can be at most one solution to equation~\eqref{eq:der_roots} in $\left(c-\frac{(n-2)\pi}{n},\frac{c}{2}\right)$. Suppose, for the sake of contradiction, that $x_{1}<x_{2}$ are two distinct solutions to \eqref{eq:der_roots} in the interval  $\left(c-\frac{(n-2)\pi}{n},\frac{c}{2}\right)$ and $\delta = x_{2}-x_{1}>0$. Then, we have
\begin{equation}\label{eq:disprove}
g_{n}'\left(x_{1}+\delta\right)-g_{n}'\left(x_{1}\right)=g_{n}'\left(c-x_{1}-\delta\right)-g_{n}'\left(c-x_{1}\right).
\end{equation}
However, by Lemma \ref{lem:der_convex}, the function $g_{n}'$ is strictly concave. Thus, for a fixed constant $x_{1}$, the function $g_{n}'\left(x_{1}+y\right)-g_{n}'\left(x_{1}\right)$ strictly decreases as $y$ increases in $\left(0,\frac{c}{2}-x_{1}\right)$ (this is an elementary application of the mean value theorem). Similarly, for the fixed constant $c-x_{1}$, the function $g'_{n}\left(c-x_{1}-y\right)-g'_{n}\left(c-x_{1}\right)$ strictly increases as $y$ increases in $\left(0,\frac{c}{2}-x_{1}\right)$. Since $x_{1}$ is a solution to equation \eqref{eq:der_roots}, we get that 
\[
g_{n}'\left(\frac{c}{2}\right)-g_{n}'\left(x_{1}\right)=g_{n}'\left(\frac{c}{2}\right)-g_{n}'\left(c-x_{1}\right),
\]
so that $y=\frac{c}{2}-x_{1}$ is a solution to equation \eqref{eq:disprove}. Thus, for any $\delta<\frac{c}{2}-x_{1}$, equation \eqref{eq:disprove} cannot be satisfied, and thus we have a contradiction.

Now, note that $$ \lim_{x\to \left(c-\frac{(n-2)\pi}{n}\right)^{+}}g_{n}'(x)=+\infty, $$ and $$ \lim_{x\to\left(\frac{(n-2)\pi}{n}\right)^{-}}g'_{n}(x)=-\infty.$$
Thus, we have two possibilities below:
\begin{enumerate}
\item Equation \eqref{eq:der_roots} has no solutions in $\left(c-\frac{(n-2)\pi}{n},\frac{c}{2}\right)$. In this case, the only local extremum of $h_{n}$ is at $x=\frac{c}{2}$, which is a local maximum. 
\item Equation~\eqref{eq:der_roots} has exactly one solution $x_{0}\in\left(c-\frac{(n-2)\pi}{n},\frac{c}{2}\right)$. In this case, $x_{0}$ and $c-x_{0}$ are local maxima of $h_{n}$, and $\frac{c}{2}$ is a local minimum of $h_{n}$.
\end{enumerate}
This concludes the argument.
\end{proof}

\begin{lemma} \label{lem:one_root}
The function $\phi_{n}:\left(0,\frac{(n-2)\pi}{n}\right)\to\mathbb{R}$, defined by $$ \phi_{n}(x)=2 g_{n} \left( \frac{x}{2} + \frac{\pi}{2}-\frac{\pi}{n}\right)-g_{n}(x)$$
has a unique root in the domain $\left(0,\frac{(n-2)\pi}{n}\right)$. 
\end{lemma}

\begin{proof}
First, note that as $x$ decreases to $0$, the function $g_{n}\left(\frac{x}{2}+\frac{\pi}{2}-\frac{\pi}{n}\right)$ approaches to $g_{n}\left(\frac{\pi}{2}-\frac{\pi}{n}\right)$ (which is finite), while the function $g_{n}(x)$ grows to $+\infty$. Thus, $$ \lim_{x\to0^{+}}\phi_{n}(x)=-\infty.$$
Now, by Lemma~\ref{lem:der_convex} and the facts that $$
\lim_{x\to0^{+}}g'_{n}(x)=\lim_{x\to\frac{(n-2)\pi}{n}^{-}}g'_{n}(x)=-\infty,$$ we see that $g''_{n}(x)$ has a unique root $x_{0}\in\left(0,\frac{(n-2)\pi}{n}\right)$. Note that the value $x_{0}$ depends only on $n$. For $\theta>x_{0}$, the function $g_{n}$ restricted on the domain ${\left(\theta,\frac{(n-2)\pi}{n}\right)}$ is strictly concave, so that $$2g_{n}\left(\frac{\theta}{2}+\frac{(n-2)\pi}{2n}\right)>g_{n}\left(\theta\right).$$
Thus, for $\theta\in\left(x_{0},\frac{(n-2)\pi}{n}\right)$, $\phi_{n}(\theta)>0$. Hence, $\phi_{n}$ has at least one root in the domain $\left(0,x_{0}\right)$. 

To see that $\phi_{n}$ has exactly one root in $\left(0,\frac{(n-2)\pi}{n}\right)$, we will show that $\phi'_{n}$ has at most one root in this domain. The result will follow by noting that 
\[
\lim_{x\to\frac{(n-2)\pi}{n}}\phi_{n}(x)=0.
\]
For the sake of contradiction, suppose $x_{1}<x_{2}$ are two zeroes of $\phi_{n}'$ in $\left(0,\frac{(n-2)\pi}{n}\right)$. Then $x_{1},x_{2}$ satisfy 
\begin{align}
\phi_{n}'\left(x\right)=g'_{n}\left(\frac{x}{2}+\frac{\pi}{2}-\frac{\pi}{n}\right)-g'_{n}(x) & =0\nonumber \\
\implies g'_{n}\left(\frac{x}{2}+\frac{\pi}{2}-\frac{\pi}{n}\right) & =g'_{n}(x).\label{eq:one_root}
\end{align}
Once again, since $g'$ is convex (by Lemma \ref{lem:der_convex}) and $g''$ has a unique root $x_{0}$, for any real number $c<g_{n}'\left(x_{0}\right)$, we see that there are exactly two solutions $\alpha$ and $\beta$ to the equation $$g'_{n}(x)=c,$$
with $0<\alpha<x_{0}<\beta<\frac{(n-2)\pi}{n}$. Moreover, since $$ \frac{x}{2}+\frac{\pi}{2}-\frac{\pi}{n}>x, $$ for $x\in\left(0,\frac{(n-2)\pi}{n}\right)$, we have that $$ x_{1}<x_{0}<\frac{x_{1}}{2}+\frac{\pi}{2}-\frac{\pi}{n}.$$
Since $x_{2}$ also satisfies \eqref{eq:one_root}, and $x_{2}>x_{1}$, we further have that 
\[
x_{1}<x_{2}<x_{0}<\frac{x_{1}}{2}+\frac{\pi}{2}-\frac{\pi}{n}<\frac{x_{2}}{2}+\frac{\pi}{2}-\frac{\pi}{n}.
\]
However, $x_{1}<x_{2}<x_{0}$ implies that $g_{n}'\left(x_{1}\right)<g_{n}'\left(x_{2}\right),$ and $x_{0}<\frac{x_{1}}{2}+\frac{\pi}{2}-\frac{\pi}{n}<\frac{x_{2}}{2}+\frac{\pi}{2}-\frac{\pi}{n}$ implies that $g_{n}'\left(\frac{x_{1}}{2}+\frac{\pi}{2}-\frac{\pi}{n}\right)>g_{n}'\left(\frac{x_{2}}{2}+\frac{\pi}{2}-\frac{\pi}{n}\right)$. Thus, $x_{1},x_{2}$ cannot both satisfy equation \eqref{eq:one_root}, and we reach our desired contradiction.  
\end{proof}

\subsection{Proof of Theorem \ref{thm:Hyp_two}}

\begin{proof}[Proof of Theorem \ref{thm:Hyp_two}]
 Suppose $P_{1}$ and $P_{2}$ are two regular hyperbolic $n$-gons, whose areas add up to $\mathrm{area}(P)$. By Gauss-Bonnet theorem, if $\theta_{1},\theta_{2}$ and $\theta$ are the interior angles of $P_{1},P_{2}$ and $P$ respectively, then we have $$ \theta_{1}+\theta_{2}=\theta+\frac{(n-2)\pi}{n}.$$
For a fixed $\theta\in\left(0,\frac{(n-2)\pi}{n}\right)$, consider $c=\theta+\frac{(n-2)\pi}{n}$. Then the expression $\mathrm{perim}\left(P_{1}\right)+\mathrm{perim}\left(P_{2}\right)$ can be written as $$ 2n\left(g_{n}\left(\theta_{1}\right)+g_{n}\left(c-\theta_{1}\right)\right)=2nh_{n}\left(\theta_{1}\right),$$
where $g_n$ is defined in equation~\eqref{eq:g} and $h_{n}$ is defined as in Lemma \ref{lem:local_min}. Thus, it suffices to prove that the inequality 
\begin{equation}\label{ineq:hyp_two}
h_{n}\left(\theta_{1}\right)\geq g_{n}\left(\theta\right)
\end{equation}
holds for every $\theta_{1}\in\left(c-\frac{(n-2)\pi}{n},\frac{(n-2)\pi}{n}\right)$, for a given value of $\theta$. 

Firstly, note that in the limiting cases, when $\theta_{1}$ approaches either ends of its domain, then $h_{n}\left(\theta_{1}\right)$ tends to $g_{n}\left(\theta\right)$ and we have equality. In the domain $\left(c-\frac{(n-2)\pi}{n},\frac{(n-2)\pi}{n}\right)$, Lemma~\ref{lem:local_min} shows that the only local minimum for $h_{n}$ occurs at $\theta_{1}=\frac{c}{2}$. Thus, inequality \eqref{ineq:hyp_two} is satisfied for every $\theta_{1}\in\left(c-\frac{(n-2)\pi}{n},\frac{(n-2)\pi}{n}\right)$ if and only if $$ h_{n} \left(\frac{c}{2}\right) = 2g_{n}\left(\frac{\theta}{2}+\frac{\pi}{2}-\frac{\pi}{n}\right)\geq g_{n}\left(\theta\right).$$
Let $\Theta$ denote the unique root to the function $$\phi_{n}(x)=2g_{n}\left(\frac{x}{2}+\frac{\pi}{2}-\frac{\pi}{n}\right)-g_{n}\left(x\right),$$ in $\left(0,\frac{(n-2)\pi}{n}\right)$, as demonstrated in Lemma~\ref{lem:one_root}. Note that $\Theta$ is determined solely by $n$. Then (using the computations in Lemma~\ref{lem:one_root}), $\phi_{n}(x)\geq0$ if and only if $x\in\left(\Theta,\frac{(n-2)\pi}{n}\right)$, with the inequality being strict when $x\in\left(\Theta,\frac{(n-2)\pi}{n}\right)$. This is precisely what we wanted to prove. 
\end{proof}

As a direct corollary, we now prove Theorem~\ref{thm:Hyp}.

\begin{proof}[Proof of Theorem \ref{thm:Hyp}]
Suppose for some $k\geq3$, we have a configuration of $k$ regular hyperbolic $n$-gons $P_{1},\dots,P_{k}$, whose total area equals $\mathrm{area}(P)$. If $\theta<\Theta$, then Theorem \ref{thm:Hyp_two} automatically gives us a configuration of two congruent regular $n$-gons $Q_{1}$ and $Q_{2}$, whose total area equals $\mathrm{area}(P)$ and whose total perimeter is less than the perimeter of $P$. 

Suppose $\theta\geq\Theta$. This is equivalent to the condition that $$ \mathrm{area}(P)\leq\left(n-2\right)\pi-n\Theta,$$
by Gauss-Bonnet theorem. Construct regular hyperbolic $n$-gons $P_{1}',\dots,P_{k}'$, such that $$ \mathrm{area}\left(P_{i}'\right)=\sum_{j=1}^{i}\mathrm{area}\left(P_{j}\right),$$
for each $i=1,\dots,k$. We get that 
\begin{align*}
\mathrm{area}\left(P_{1}'\right)<\mathrm{area}\left(P_{2}'\right)<\dots & <\mathrm{area}\left(P_{k}'\right)=\mathrm{area}(P)\\
 & \leq\left(n-2\right)\pi-n\Theta.
\end{align*}
Thus, we can repeatedly apply Theorem \ref{thm:Hyp_two} to obtain the sequence of inequalities 
\begin{align*}
\mathrm{perim}\left(P_{1}\right)+\mathrm{perim}\left(P_{2}\right) & \geq\mathrm{perim}\left(P_{2}'\right)\\
\mathrm{perim}\left(P_{2}'\right)+\mathrm{perim}\left(P_{3}\right) & \geq\mathrm{perim}\left(P_{3}'\right)\\
\vdots\\
\mathrm{perim}\left(P_{k-1}'\right)+\mathrm{perim}\left(P_{k}\right) & \geq\mathrm{perim}\left(P\right).
\end{align*}
Summing these inequalities together gives us the desired result. Note that when $\theta>\Theta$, each of these inequalities is strict, and so would their sum be.
\end{proof}

\section{Further thoughts}\label{sec:Further}

Here are some possible applications / lines of thought that follow from this work.
\begin{enumerate}
\item The proof of Theorem \ref{thm:Euc} in the Euclidean case, suggests a link between the Pythagorean theorem and the isoperimetric inequality for disconnected regions. The cosine law in spherical and hyperbolic geometry helps relate the length of a side of a triangle with the length of the other two sides. This motivates a possible way to prove Theorems \ref{thm:Sph} and \ref{thm:Hyp}, by associating the perimeters of two regular polygons in a configuration and the perimeter of the regular polygon with the same total area to three sides of a triangle.
\item Proposition \ref{prop:weak} shows that when $\mathbb{M}=\mathbb{R}^{2}$ or $\mathbb{S}^{2}$, the only configuration of polygons that minimizes the total perimeter (for a given total area) is one consisting of a single regular polygon. Theorem \ref{thm:Hyp} shows that when $\mathbb{M}=\mathbb{H}^{2}$, the configuration consisting of a single regular polygon isn't the unique configuration minimizing the total perimeter when $\theta=\Theta$, and that this does not minimize total perimeter when $\theta<\Theta$. One may thus work on finding exactly the total perimeter minimizing configurations in $\mathbb{H}^{2}$, for $\theta\leq\Theta$. 
\item One may try to extend Proposition \ref{prop:weak} and Theorem \ref{thm:Hyp} to configurations of polygons with a general number of sides. Notably in this direction, for a fixed area, the perimeter of a regular $n$-gon is monotonically decreasing with $n$ (for $\mathbb{M}=\mathbb{R}^{2},\mathbb{S}^{2}$ or $\mathbb{H}^{2}$). Thus, the corresponding bounds are automatically satisfied for polygons with ``at most $n$ sides" instead of ``polygons with $n$ sides". However, considering configurations of polygons with number of sides satisfying some relation could lead to tighter bounds.
\item As in \cite{Csi}, one may try to extend the contents of Proposition \ref{prop:weak} and Theorem \ref{thm:Hyp} to polygons bounded by lines of constant geodesic curvature.
\end{enumerate}

{}

\end{document}